\documentclass[11pt]{amsart}

\usepackage{xcolor}
\usepackage{mathrsfs,color}
\usepackage{amsmath,amssymb,amsthm,tikz}

\usepackage[pagewise]{lineno}

\usepackage{tikz}
\usetikzlibrary{cd}
\usetikzlibrary{arrows, petri, topaths}

\input xy
\xyoption{all}
\newtheorem{theorem}{Theorem}[section]

\newtheorem{thm}[theorem]{Theorem}
\newtheorem{cor}[theorem]{Corollary}
\newtheorem{lemm}[theorem]{Lemma}
\newtheorem{prop}[theorem]{Proposition}

\theoremstyle{definition}

\newtheorem{definition-theorem}[theorem]{Definition-Theorem}
\newtheorem{defi}[theorem]{Definition}
\newtheorem{remk}[theorem]{Remark}
\newtheorem{exam}[theorem]{Example}

\newcommand{\twosilt}{\mbox{\rm 2-silt}\hspace{.01in}}
\newcommand{\twotilt}{\mbox{\rm 2-tilt}\hspace{.01in}}

\newcommand{\DD}{\mathsf{D}}

\newcommand{\KKb}{\mathsf{K}^{\rm b}}
\newcommand{\CCb}{\mathsf{C}^{\rm b}}

\newcommand{\K}{\mathsf{K}}

\newcommand{\Aut}{\operatorname{Aut}\nolimits}
\newcommand{\La}{\Lambda}
\newcommand{\Ga}{\Gamma}
\newcommand{\stilt}{\mbox{\rm s$\tau$-tilt}\hspace{.01in}}

\newcommand{\add}{\mathsf{add}\nolimits}
\newcommand{\proj}{\mathsf{proj}}

\newcommand{\id}{\operatorname{id}\nolimits}

\newcommand{\Hom}{\operatorname{Hom}\nolimits}
\newcommand{\Endm}{\operatorname{End}\nolimits}
\newcommand{\Ext}{\operatorname{Ext}\nolimits}

\newcommand{\rad}{\operatorname{rad}\nolimits}

\def\Im{\mathop{\mathrm{Im}}\nolimits}
\def\Ker{\mathop{\mathrm{Ker}}\nolimits}

\def\mod{\mathsf{mod}}
\def\Mod{\mathsf{Mod}}
\def\fp{\mathsf{fp}}

\def\id{\mathrm{id}}
\def\RHom{\mathop{\mathbf R\mathrm{Hom}}\nolimits}

\newcommand{\Lotimes}{\mathop{{\otimes}^\mathbf{L}_\Lambda}\nolimits}
\newcommand{\fm}{\mathfrak{m}}
\newcommand{\cC}{\mathcal{C}}

\def\add{{\mathsf{add}}}

\def\op{\mathop{\mathrm{op}}\nolimits}

\newcommand{\xto}[1]{\xrightarrow{#1}}
\newcommand{\EE}{\mathcal{E}}

\begin{document}
\title{Two-term tilting complexes for preprojective algebras of non-Dynkin type}
\author{Yuta Kimura}
\address{Fakult\"{a}t f\"{u}r Mathematik, Universit\"{a}t Bielefeld, 33501 Bielefeld, Germany} 
\email{ykimura@math.uni-bielefeld.de}

\author{Yuya Mizuno}
\address{Faculty of Liberal Arts and Sciences, 	Osaka Prefecture University,
  1-1 Gakuen-cho, Naka-ku, Sakai, Osaka 599-8531, Japan}
\email{yuya.mizuno@las.osakafu-u.ac.jp}
\thanks{The first-named author is supported by the Alexander von Humboldt Stiftung/Foundation in the framework of the Alexander von Humboldt Professorship endowed by the Federal Ministry of Education and Research.
}

\thanks{}

\begin{abstract}
In this paper, we study two-term tilting complexes for preprojective algebras of non-Dynkin type.
We show that there exist two families of two-term tilting complexes, which are respectively parameterized by the elements of the corresponding Coxeter group. 
Moreover, we provide the complete classification in the case of affine type by showing that any two-term silting complex 
belongs one of them. For this purpose, we also discuss the Krull-Schmidt property for the homotopy category of finitely generated projective modules over a complete ring.
\end{abstract}
\maketitle


\section{Introduction}

Let $\Delta$ be a finite graph without loops, $W$ the Coxeter group of $\Delta$ and 
$\La$ the preprojective algebra of $\Delta$. 
Recently, a strong connection between the representation theory of $\La$ and  $W$ has been discovered, and this link allows us to study the category of $\La$-modules 
in terms of combinatorics of $W$. 
To explain this relationship more precisely, we give the following notations. 

Let $I_i:=\La(1-e_i)\La$ be the two-sided ideal of $\La$, where $e_i$ is the idempotent corresponding to $i\in\Delta_0$, and 
$\langle I_1,\ldots,I_n\rangle:=\{I_{i_1}\cdots I_{i_l}\ |\  l\geq 0, i_1,\ldots,i_l\in \Delta_0\}$. 
Then, by \cite{BIRS,IR1}, we have a bijection 
$$W\ni w=s_{i_1}\cdots s_{i_l} \mapsto I_w:=I_{i_1}\cdots I_{i_l}\in \langle I_1,\ldots,I_n\rangle,$$ 
where $s_{i_1}\cdots s_{i_l}$ is a reduced expression of $w$. 
The object $I_w$ plays a quite important role not only in the study of the category of $\La$-modules but also from the viewpoint of categorification of cluster algebras, for example \cite{AM,AIRT,BIRS,GLS,IRRT,Ki1,Ki2,L,M1,M2,SY}.  
Among others,  the situation is particularly nice if $\Delta$ is a Dynkin graph. 
In this case, $\langle I_1,\ldots,I_n\rangle$ can be identified with the set $\stilt\La$ of support $\tau$-tilting $\La$-modules, and the above map provides a bijection 
$W\to\stilt\La$ \cite{M1}. 
Then 
this map induces a poset isomorphism (defined by the weak order) 
and an action of simple generators of $W$ can be interpreted as \emph{mutation} of $\stilt\La$. 
This fact, together with general properties of support $\tau$-tilting modules, yield a comprehensive study of various important objects in the category such as torsion classes, silting complexes and so on (we refer to \cite{AIR,BY,IR2} for the background of $\tau$-tilting theory). 

One of the main motivation of this paper is to generalize this result to non-Dynkin cases.
Namely, we study two-term tilting and silting complexes of preprojective algebras of non-Dynkin type.
We note that Auslander-Reiten duality $\tau$ is defined for finite dimensional algebras, while preprojective algebras of non-Dynkin type are infinite dimensional.
Therefore, it is reasonable to study the set of two-term silting complexes,  
which is naturally in bijection with the set of support $\tau$-tilting modules if an algebra is finite dimensional.

In \cite{IR1,BIRS}, it is shown that $I_w$ is a (classical) tilting module if $\Delta$ is non-Dynkin. 
In this paper, we improve this result from the viewpoint of $\tau$-tilting theory and silting theory.
Following \cite{BIRS}, we first provide two initial families of two-term tilting complexes, 
which are respectively parameterized by  the elements of  the corresponding Coxeter group.
Moreover, in the case of affine type, we show that any two-term silting complex belongs one of them.

Our main results are summarized as follows (see Propositions \ref{P_w}, \ref{R_w}, \ref{nointersection} and Theorem \ref{classify} for some notation).

\begin{thm}\label{main}
Let $\Delta$ be a non-Dynkin graph, $\La=\La_\Delta$ the complete preprojective algebra of $\Delta$ and $W=W_\Delta$ the Coxeter group of $\Delta$. 
We denote by $\twotilt\La$ the set of isomorphism classes of basic two-term tilting complexes of $\KKb(\proj\La)$ $($Note that the sets $W$ and $\twotilt\La$ have  natural partial orderings, see section \ref{section-preproj-alg}$)$.
\begin{itemize}
\item[(a)] There are order-reversing injection and order-preserving injection
$$\phi :W\to \twotilt\La,\ \ w\mapsto P_w\ \  \textnormal{and}\ \ \phi^* :W\to \twotilt\La,\ \ w\mapsto R_w,$$ 
and
$\{P_w\}_{w\in W} \bigcap \{R_w\}_{w\in W} = \emptyset$.

\item[(b)]Moreover, assume that $\Delta$ is affine type. Then 
$$\twosilt\La = \{P_w\}_{w\in W} \coprod \{R_w\}_{w\in W},$$
where $\twosilt\La$ is the set of isomorphism classes of basic two-term silting complexes of $\KKb(\proj\La)$.
In particular, any two-term silting complex of $\KKb(\proj\La)$ is a tilting complex.
\end{itemize}
\end{thm}


For this result, we use the notion of \emph{$g$-vectors} and, using this notion, we define the cone for each  two-term tilting complex.
The chamber structure of these cones plays a very important role in the study of the mutation behavior of two-term tilting complexes and 
several applications have been discovered (for example \cite{A,DIJ,BST,IW2, Y}). 
In particular, density of cones can be regarded as a natural generalization of $\tau$-tilting finiteness and almost all elements in the Grothendieck group can be realized as $g$-vectors of two-term presilting complexes.  
In the case of  the preprojective algebra of non-Dynkin type, we observe that each cone of a two-term tilting complex corresponds to a Weyl chamber of the Coxeter group $W$ (Theorem \ref{birs geometric}).
Then we investigate a chamber structure of them (Proposition \ref{cover}), and together with the technique of $\tau$-tilting theory \cite{DIJ}, we prove the theorem.

Since a preprojective algebra of non-Dynkin type is not finite dimensional, the Krull-Schmidt property is quite non-trivial. 
To discuss the above result, we show that the homotopy category of finitely generated projective modules over a complete ring with a condition (F) has the Krull-Schmidt property, see section \ref{Krull-Schmidt property}.



One of the remarkable property of preprojective algebras of non-Dynkin type is that a family of complexes obtained  from $\La$ by mutation is always tilting. 
Therefore it is hard to find even an example of a silting complex which is not tilting. 
Thus, for a  preprojective algebra $\La$ of non-Dynkin graph, it is natural to ask that any silting complex of $\KKb(\proj\La)$ is tilting or not. 
Though we do not know the answer in general, 
in the case of affine graph, a positive answer was informed us by Osamu Iyama.  
We are grateful to him for agreeing to include his proof as an appendix.
\\

\textbf{Acknowledgements.}
The authors are grateful to Osamu Iyama for useful discussions and allowing us to write his result in the appendix.

\section{Preprojective algebra of non-Dynkin type}
\label{section-preproj-alg}

Throughout this paper, let $K$ be an algebraically closed field. 
Modules mean left modules. 
Let $Q=(Q_{0}, Q_{1}, s, t)$ be a finite acyclic quiver.
The {\it double quiver} $\overline{Q}=(\overline{Q}_{0},\overline{Q}_{1},s,t)$ of $Q$ is defined by $\overline{Q}_{0}=Q_{0}$, $\overline{Q}_{1}=Q_{1}\sqcup \{\alpha^{\ast}: t(\alpha)\to s(\alpha) \mid \alpha \in Q_{1}\}$. 
For two arrows $\alpha$ and $\beta$ of $Q$ such that $s(\alpha)=t(\beta)$, we denote by $\alpha\beta$ the composite of them.

We first recall the notion of preprojective algebras.
 
\begin{defi}\label{Preprojective algebras}
Let $Q$ be a finite acyclic quiver whose underlying graph is $\Delta$.
The {\it preprojective algebra} of $Q$ is defined  as follows 
\begin{align*}
	K\overline{Q}/\langle \displaystyle\sum\limits_{\alpha\in Q_1} \alpha\alpha^{\ast}-\alpha^{\ast}\alpha \rangle . 
	\end{align*}
The \emph{(complete) preprojective algebra} $\La$ of $Q$ is the completion of the preprojective algebra by the arrow ideal.
Since $\La$ is independent of the orientation of $Q$, we write $\La=\La_{\Delta}$.
\end{defi} 

In the rest of this section, 
we fix a non-Dynkin graph $\Delta$ (i.e., $\Delta$ is not type $\mathbb{A},\mathbb{D}$ and $\mathbb{E}$)
and $\La=\La_\Delta$. 
We denote by $\proj\La$  the category of finitely generated projective $\La$-modules and $\KKb(\proj\La)$ the homotopy category  of bounded complexes of $\proj\La$. 
An $\La$-module $M$ is called \emph{finitely presented} if there exists an exact sequence 
$P \to Q \to M \to 0$ with $P, Q\in\proj \La$.
We denote by $\fp \La$ the full subcategory of the category of $\La$-modules consisting of finitely presented modules.

Recall that an additive category is called \emph{Krull-Schmidt} if each object is a finite direct sum of objects such that whose endomorphism algebras are local.
Since the complete preprojective algebra $\La$ is a complete ring satisfying the condition (F) (Example \ref{exam-pseudo-compact-ring} (iii)), we have the Krull-Schmidt property as follows.

\begin{thm}\label{KS-thm}
$\KKb(\proj\La)$ and $\fp \La$ are Krull-Schmidt categories.
\end{thm}
Theorem \ref{KS-thm} follows from the general property of complete rings satisfying the condition (F), which is discussed 
in section \ref{Krull-Schmidt property}.

Next we recall the definitions of tilting modules and tilting complexes.
Recall that an object $X$ of a Krull-Schmidt category is called \emph{basic} if the multiplicity of each indecomposable direct summand of $X$ is one.

\begin{defi}\label{defi}
\begin{itemize}
\item[(a)] 
A $\La$-module $T$ is called a \emph{tilting module} (of projective dimension at most one) if it satisfies
\begin{itemize}
\item[(i)]
there exists an exact sequence $0 \to P_{1} \to P_{0} \to T \to 0$ with a finitely generated projective $\La$-modules $P_{i}$,
\item[(ii)]
$\Ext_{\La}^{1}(T,T)=0$, and
\item[(iii)]
there exists an exact sequence $0 \to \La \to T_{0} \to T_{1} \to 0$ with $T_{0}, T_{1}$ in $\add T$.
\end{itemize}

\item[(b)] 
Let $S, T$ be tilting $\La$-modules.
Then we write $S\geq T$ if $\Ext_{\La}^{1}(S,T)=0$.
We regard the set of isomorphism classes of basic tilting modules as a poset by this order \cite{HU, RS}.

\item[(c)] 
A complex $T\in\KKb(\proj\La)$ is called a \emph{tilting complex} if it satisfies
\begin{itemize}
\item[(i)]
$\Hom_{\KKb(\proj\La)}(T,T[i])=0$ for any $i\neq 0$, and
\item[(ii)]
$T$ generates $\KKb(\proj\La)$.
\end{itemize}

\item[(d)] 
Let $S,T$ be tilting complexes.
We write $S\geq T$ if $\Hom_{\KKb(\proj\La)}(S, T[i])=0$ for any $i>0$.
We can regard the set of isomorphism classes of basic tilting complexes as a poset by this order \cite[Definition 2.12]{AI}.
\end{itemize}
\end{defi}
In this paper, we study two-term tilting complexes. 
Recall that a complex $P=(P^{i}, d^{i})$ in $\KKb(\proj\La)$ is a \emph{two-term} complex if $P^{i}=0$ for all $i\neq -1, 0$. 


The {\it Coxeter group} $W=W_{\Delta}$ of $\Delta$ is the group generated by the set  $\{\,s_{i} \mid i\in \Delta_0\,\}$ with relations $s_{i}^2=1$, $s_{i}s_{j}=s_{j}s_{i}$ if there exist no edge between $i$ and $j$, and $s_{i}s_{j}s_{i}=s_{j}s_{i}s_{j}$ if there exists exactly one edge between $i$ and $j$.
If a word $s_{i_1}\cdots s_{i_l}$ represents an element $w\in W$, then we write $w=s_{i_1}\cdots s_{i_l}$ and say that $s_{i_1}\cdots s_{i_l}$ is an \emph{expression} of $w$.
Let $s_{i_1}\cdots s_{i_l}$ be an expression of $w$. 
If $l$ is minimal among all expressions of $w$, then we say that $s_{i_1}\cdots s_{i_l}$ is a \emph{reduced expression} of $w$, $l$ is called a \emph{length} of $w$ and we write $\ell(w)=l$.

For $v, w\in W$, we write $v\leq w$ if $\ell(v^{-1}w)=\ell(w)-\ell(v)$ holds.
We call this $\leq$ the \emph{(right) weak order} and regard $W$ as a poset by this order.

We denote by $I_i$ the two-sided ideal of $\La$ generated by $1-e_i$, where $e_i$ is the primitive idempotent of $\La$ for $i\in \Delta_0$.  
We denote by $\langle I_1,\ldots,I_n\rangle$ the set of ideals of $\La$ which can be written as 
$I_{i_1}\cdots I_{i_l}$ for some $l\geq 1$ and $i_1,\ldots,i_l\in \Delta_0$, where $I_{\id}:=\La$.
Then the following result was shown.

\begin{theorem}\cite[Theorem III.1.9, 1.13]{BIRS}\label{birs}
\begin{itemize}
\item[(a)] There exists a bijection $W\to\langle I_1,\ldots,I_n\rangle$. 
It is given by $w\mapsto I_w =I_{i_1}\cdots I_{i_l}$ for any reduced 
expression $w=s_{i_1}\cdots s_{i_l}$.
\item[(b)] $I_w$ is a tilting $\La$-module and a tilting $\La^{op}$-module for any $w\in W$. 
\item[(c)] The map $w\mapsto I_w$ gives a poset anti-isomorphism 
between $W$ and $\langle I_1,\ldots,I_n\rangle$.
\end{itemize}
\end{theorem}

By Theorems \ref{KS-thm} and \ref{birs}, there exists a minimal projective resolution of $I_w$ as a $\La$-module
$$\xymatrix@C10pt@R5pt{0\ar[r]& P^1_w\ar[r]^{f}& P^0_w \ar[r]^{}& I_w\ar[r]&0}$$ with $f\in \rad_\La(P_w^1,P_w^0)$. 
Then put
$$P_w :=(\xymatrix@C10pt@R5pt{\cdots\ar[r]&0\ar[r]& P^1_w\ar[r]^{f}& P^0_w \ar[r]^{}&0\ar[r]&\cdots})\in\KKb(\proj\La),$$ 
which is a two-term complex. 
Then we have the following proposition.

\begin{prop}\label{P_w}
\begin{itemize}
\item[(a)] $P_w$ is a two-term tilting complex of $\KKb(\proj\La)$ for any $w\in W$.
\item[(b)] The map $w\mapsto P_w$ gives a poset anti-isomorphism 
between $W$ and  $\{P_w\}_{w\in W}$.
\end{itemize}
\end{prop}

\begin{proof}
(a) Since $I_{w}$ is a tilting $\La$-module, $P_{w}$ is a tilting complex.

(b)
We have two bijections
$$W \longrightarrow \langle I_1,\ldots,I_n\rangle \longrightarrow \{P_w\}_{w\in W},$$
where the first map is given by $w \mapsto I_{w}$ as Theorem \ref{birs} and the second one is given by $I_{w} \mapsto P_{w}$ as above.
By Theorem \ref{birs}, the first map is a poset anti-isomorphism.
By the definitions of orderings of tilting modules and tilting complexes, the second one is a poset isomorphism.
Therefore the map $w \mapsto P_{w}$ gives a poset anti-isomorphism.
\end{proof}

Since $I_w$ is also a tilting $\La^{\op}$-module, we can similarly take a minimal projective resolution of $I_{w}$ as a $\La^{\op}$-module
$\xymatrix@C10pt@R5pt{0\ar[r]& Q^1_w\ar[r]^{g}& Q^0_w \ar[r]& I_w\ar[r]&0}$ with $g\in \rad_{\La^{\op}}(Q_w^1,Q_w^0).$
Then put
$$Q_w :=(\xymatrix@C10pt@R5pt{\cdots\ar[r]&0\ar[r]& Q^1_w\ar[r]^{g}& Q^0_w \ar[r]^{}&0\ar[r]&\cdots}) \in \KKb(\proj\La^{\op}),$$
$$R_w :=\Hom_{\La^{\op}}(Q_w,\La[1]).$$
Then we have $R_w\in\KKb(\proj\La)$ and this is also a two-term complex.

\begin{prop}\label{R_w}
\begin{itemize}
\item[(a)]  $R_w$ is a two-term tilting complex  of $\KKb(\proj\La)$ for any $w\in W$.
\item[(b)] The map $w\mapsto R_w$ gives a poset isomorphism 
between $W$ and $\{R_w\}_{w\in W}$.
\end{itemize}
\end{prop}

\begin{proof}
From a duality $\Hom_\La(-,\La):\K^{\rm{b}}(\proj\La)\overset{\simeq}{\to}\K^{\rm{b}}(\proj\La^{\op}),$
we have the assertion. 
\end{proof}

Then we have the following result. 

\begin{prop}\label{nointersection}We have 
$\{P_w\}_{w\in W} \bigcap \{R_v\}_{v\in W} = \emptyset$. 
In particular, there exist two different families of tilting complexes of $\KKb(\proj\La)$.
\end{prop}

\begin{proof}

Assume that $P_w=R_v$ holds for some $v,w\in W$.

Let $g : Q_{v}^{1} \to Q_{v}^{0}$ be a minimal projective resolution of $I_{v}$ as a $\La^{\op}$-module.
By applying $\Hom_{\La^{\op}}(-,\La)$ to $g$, we have an exact sequence $$0 \to \Hom_{\La^{\op}}(I_{v}, \La) \to \Hom_{\La^{\op}}(Q_{v}^{0}, \La) \xto{g^{\ast}} \Hom_{\La^{\op}}(Q_{v}^{1},\La).$$
This implies that $H^{-1}(R_v)=\Hom_{\La^{\op}}(I_{v}, \La)\neq 0$.

On the other hand, since $P_{w} = (P_{w}^{1} \to P_{w}^{0})$ is a minimal projective resolution of $I_{w}$, we have $H^{-1}(P_{w})=0$, which is a contradiction.
\end{proof}



\section{Preprojective algebras of affine type}
In this section, we assume that $\La$ is a complete preprojective algebra of an affine graph $\Delta$, that is, $\Delta$ is one of the following graphs.

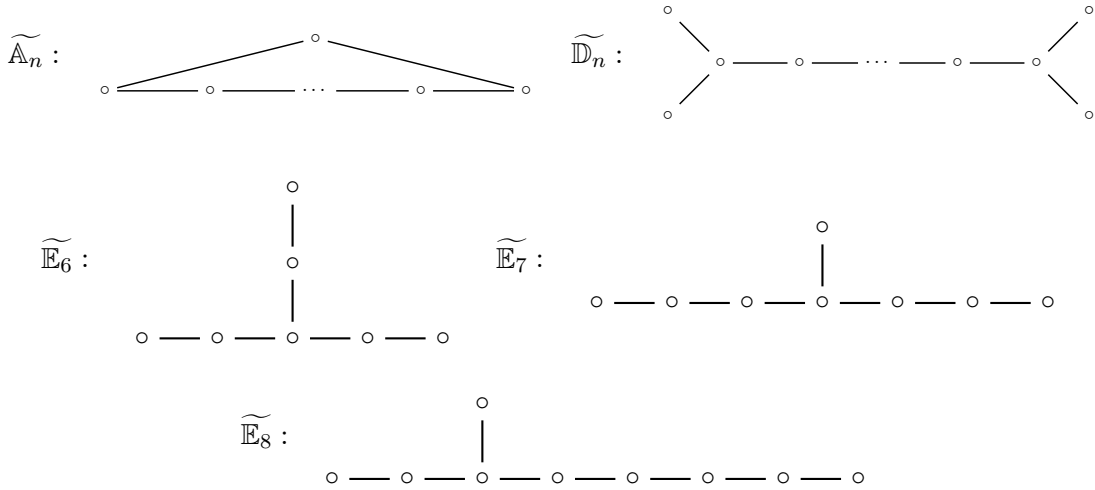
\begin{equation*}
\begin{array}{clcl}
\widetilde{\mathbb{A}_{n}}: &
\scalebox{0.7}
{
	\begin{tikzpicture}[baseline=15]
	\node(1)at(0,0){$\circ$};
	\node(0)at(4,1){$\circ$};
	\node(2)at(2,0){$\circ$};
	\node(cdots)at(4,0){$\cdots$};
	\node(n-1)at(6,0){$\circ$};
	\node(n)at(8,0){$\circ$};
	\draw[thick, -] (1)--(2);
	\draw[thick, -] (2)--(cdots);
	\draw[thick, -] (cdots)--(n-1);
	\draw[thick, -] (n-1)--(n);
	\draw[thick, -] (n)--(0);
	\draw[thick, -] (1)--(0);
	\end{tikzpicture}
} &
\widetilde{\mathbb{D}_{n}}: &
\scalebox{0.7}
{
\begin{tikzpicture}[baseline=0]
\node(01+)at(0,1){$\circ$};
\node(01-)at(0,-1){$\circ$};
\node(1)at(1,0){$\circ$};
\node(2)at(2.5,0){$\circ$};
\node(cdots)at(4,0){$\cdots$};
\node(n-1)at(5.5,0){$\circ$};
\node(n)at(7,0){$\circ$};
\node(n1+)at(8,1){$\circ$};
\node(n1-)at(8,-1){$\circ$};
\draw[thick, -] (1)--(2);
\draw[thick, -] (2)--(cdots);
\draw[thick, -] (cdots)--(n-1);
\draw[thick, -] (n-1)--(n);
\draw[thick, -] (01+)--(1);
\draw[thick, -] (01-)--(1);
\draw[thick, -] (n1+)--(n);
\draw[thick, -] (n1-)--(n);
\end{tikzpicture}
}
\end{array}
\end{equation*}

\begin{equation*}
\begin{array}{clcl}
\widetilde{\mathbb{E}_{6}}: &
\scalebox{1}
{
\begin{tikzpicture}[baseline=-30]
\node(1)at(0,0){$\circ$};
\node(2)at(0,-1){$\circ$};
\node(3)at(-2,-2){$\circ$};
\node(4)at(-1,-2){$\circ$};
\node(5)at(0,-2){$\circ$};
\node(6)at(1,-2){$\circ$};
\node(7)at(2,-2){$\circ$};
\draw[thick, -] (1)--(2);
\draw[thick, -] (2)--(5);
\draw[thick, -] (3)--(4);
\draw[thick, -] (4)--(5);
\draw[thick, -] (5)--(6);
\draw[thick, -] (6)--(7);
\end{tikzpicture} 
}&
\widetilde{\mathbb{E}_{7}}: &
\scalebox{1}
{
\begin{tikzpicture}[baseline=-15]
\node(1)at(0,0){$\circ$};
\node(2)at(-3,-1){$\circ$};
\node(3)at(-2,-1){$\circ$};
\node(4)at(-1,-1){$\circ$};
\node(5)at(0,-1){$\circ$};
\node(6)at(1,-1){$\circ$};
\node(7)at(2,-1){$\circ$};
\node(8)at(3,-1){$\circ$};
\draw[thick, -] (1)--(5);
\draw[thick, -] (2)--(3);
\draw[thick, -] (3)--(4);
\draw[thick, -] (4)--(5);
\draw[thick, -] (5)--(6);
\draw[thick, -] (6)--(7);
\draw[thick, -] (7)--(8);
\end{tikzpicture}
}
\end{array}
\end{equation*}

\begin{equation*}
\begin{array}{cl}
\widetilde{\mathbb{E}_{8}}: &
\begin{tikzpicture}[baseline=-15]
\node(1)at(0,0){$\circ$};
\node(2)at(-2,-1){$\circ$};
\node(3)at(-1,-1){$\circ$};
\node(4)at(-0,-1){$\circ$};
\node(5)at(1,-1){$\circ$};
\node(6)at(2,-1){$\circ$};
\node(7)at(3,-1){$\circ$};
\node(8)at(4,-1){$\circ$};
\node(9)at(5,-1){$\circ$};
\draw[thick, -] (1)--(4);
\draw[thick, -] (2)--(3);
\draw[thick, -] (3)--(4);
\draw[thick, -] (4)--(5);
\draw[thick, -] (5)--(6);
\draw[thick, -] (6)--(7);
\draw[thick, -] (7)--(8);
\draw[thick, -] (8)--(9);
\end{tikzpicture}
\end{array}
\end{equation*}

\if()
$$\xymatrix@C25pt@R10pt{&&\circ \ar@{-}[rrdd]&&\\
&&&&\\
\circ\ar@{-}[rruu] \ar@{-}[r] &\circ \ar@{-}[r]&\cdots\ar@{-}[r]&\circ \ar@{-}[r]&  \ar@{-}[l]\circ} $$  
$$\xymatrix@C20pt@R15pt{&&&\circ&\\
\circ \ar@{-}[r] &\circ \ar@{-}[r]
&\cdots\ar@{-}[r]&\circ \ar@{-}[r]\ar@{-}[u]
& \circ} $$ 

$$\xymatrix@C15pt@R10pt{
&&\circ&&\\
&&\ar@{-}[u]\circ&&\\
\circ \ar@{-}[r] &\circ \ar@{-}[r]
&\circ\ar@{-}[r]\ar@{-}[u]&\circ \ar@{-}[r]
& \circ} $$  
$$\xymatrix@C15pt@R15pt{
&&&\circ&&&\\
\circ \ar@{-}[r]&\circ \ar@{-}[r] &\circ \ar@{-}[r]
&\circ\ar@{-}[r]\ar@{-}[u]&\circ \ar@{-}[r]
& \circ \ar@{-}[r]&\circ} $$

$$\xymatrix@C15pt@R15pt{
&&\circ&&&&&\\
\circ \ar@{-}[r]&\circ \ar@{-}[r] &\circ \ar@{-}[r]\ar@{-}[u]
&\circ\ar@{-}[r]&\circ \ar@{-}[r]
& \circ \ar@{-}[r]&\circ\ar@{-}[r]&\circ} $$ 
\fi

Recall that a complex $T\in\KKb(\proj\La)$ is called a \emph{silting} complex if we replace the condition (i) of Definition \ref{defi} (c) to (i') : $\Hom_{\KKb(\proj\La)}(T,T[i])=0$ for any $i> 0$.
We denote by $\twosilt\La$ (respectively, $\twotilt\La$) the set of isomorphism classes of basic two-term silting complexes (respectively, tilting complexes) of $\KKb(\proj\La)$.
Our aim is to show the following result and classify all two-term tilting complexes as follows. 

\begin{thm}\label{classify}
\begin{itemize}
\item[(a)] We have $\twosilt\La=\twotilt\La$.
\item[(b)] We have
$$\twotilt\La = \{P_w\}_{w\in W} \coprod \{R_w\}_{w\in W}.$$
\end{itemize}
\end{thm}

In the rest of this section, we will give a proof of Theorem \ref{classify}.

\begin{defi}\cite{BB}
Let $V$ be a real vector space of dimension $n=|\Delta_0|$ with a basis $\alpha_i$ ($i\in \Delta_0$) and let $V^*$ be the dual vector space with a basis $\alpha_i^*$. 
Let $m_{ij}$ be the number of edges between the 
vertices $i$ and $j$ of $\Delta$ (note that $m_{i,j}>1$ only if $\Delta$ is type $\widetilde{\mathbb{A}}_1$). 
Then we define the \emph{geometric representation} $\sigma : W \to GL(V)$ of $W$ by 
$$\sigma_{s_i}(\alpha_j)=\alpha_j + (m_{ij}-2\delta_{ij})\alpha_i.$$

The \emph{contragradient of the geometric representation} $\sigma^* : W \to GL(V^*)$ is then defined by 
$$\sigma_{s_i}^*(\alpha_j^*) =\left\{\begin{array}{ccc}
\alpha_j^*&(i\neq j)\\
&&\\
-\alpha_j^*+\sum_{t\neq j} m_{tj}\alpha_t^*&(i= j).
\end{array}\right.$$ 
\end{defi}

Note that we have $\langle\sigma_w^*(x),\sigma_w(y)\rangle = \langle x,y \rangle$, where $\langle x,y \rangle$ is the canonical pairing of $x\in V^*$ and $y\in V$.

\begin{remk}
By defining a label $m(i,j)$ by 
$$m(i,j):=\left\{\begin{array}{ccc}
1&(i= j)\\
3&(\exists\ \textnormal{one edge between } i\textnormal{ and }j )\\
\infty&(\exists\ \textnormal{two edges between } i\textnormal{ and }j )\\
2&(\nexists\ \textnormal{edge between } i\textnormal{ and }j)
\end{array}\right.$$
and a symmetric bilinear form $(-,-)$ on $V$ by 
$(\alpha_i,\alpha_j):=-\cos \frac{\pi}{m(i,j)},$
we have $\sigma_{s_i}(\alpha_j)=\alpha_j-2(\alpha_i,\alpha_j)\alpha_i$ which is the same notation as \cite{Hu}.
\end{remk}

Let $K_0(\La)_\mathbb{R}:=K_0(\KKb(\proj\La))\otimes_\mathbb{Z} \mathbb{R}$ and let $V^*\to K_0(\La)_\mathbb{R}$ be an isomorphism defined by $\alpha_i^*\mapsto [\La e_i]$.
For $T = T_1\oplus \cdots\oplus T_n \in \twotilt\La$, we denote by $C(T)$ the cone 
spanned by $\{[T_1],\ldots,[T_n]\}$ in $K_0(\La)_\mathbb{R}=V^*$, that is, $C(T) =\{\sum_{i=1}^na_i[T_i]\ |\ a_i\in\mathbb{R}_{\geq0}\}$.
We simply write $C_+:=C(P_\id)$ and $wC_+:=\sigma_w^*C_+$. 

Then we have the following result. 
\begin{thm}\cite{IR1,BIRS}\label{birs geometric}
\begin{itemize}
\item[(a)]
The induced isomorphism $GL(V^*) \to GL(K_0(\La)_\mathbb{R})$ satisfies 
$\sigma_w^* \mapsto [I_{w} \Lotimes -]$ for any $w\in W$.
\item[(b)] We have $C(P_w)=wC_+$. 
\end{itemize}
\end{thm}

\begin{proof}
(a) follows from \cite{IR1,BIRS} and  (a) implies 
$C(P_w) = C(I_{w} \Lotimes \La)=wC_+$. 
\end{proof}

To discuss our geometric situation more precisely, we prepare some notation. Note that we have $C_+=\{ f \in V^{\ast} \mid \langle f, \alpha_i \rangle \geq 0, \ {}^{\forall}i\in \Delta_0 \}$.
We define \emph{Tits cone} by $U:=\bigcup_{w\in W}wC_+$.
A \emph{real root} is an element of $\Phi:=\{\sigma_w(\alpha_i) \mid i\in \Delta_0, w\in W\}$, and for each real root $\alpha$, we  define a hyperplane $H_{\alpha}:=\{f\in V^{\ast} \mid \langle f, \alpha \rangle = 0\}$.
Then a \emph{Weyl chamber} is defined as a connected component of $U\setminus \bigcup_{\alpha\in \Phi}H_{\alpha}$.
It is well known that there exists a bijection between the set of Weyl chambers of $W$ and the elements of the Coxeter group $W$ (see \cite[Theorem 5.13]{Hu} for instance).

The following corollary is a direct consequence of Proposition \ref{birs} and Theorem \ref{birs geometric}, which gives a direct connection between two-term tilting complexes $P_w$ over $\La$ and Weyl chambers.
\begin{cor}\label{cor-W-I-chamber}
The cones of two-term tilting complexes $\{P_w\}_{w\in W}$ coincide with the closure of the Weyl chambers of $W$. 
In particular, there exist bijections between 
the set $\{P_w\}_{w\in W}$ and the set of Weyl chambers of $W$.
\end{cor}

Let $C_-:=C(R_\id)=-C_+$. 
Then we have the following result, which follows from Theorem \ref{birs geometric} and the standard fact about the affine Weyl fans  (see \cite{Hu} for example). 

\begin{prop}\label{cover}We have 
$$\overline{\bigcup_{w\in W}wC_+\cup\bigcup_{w\in W}wC_-} = V^*,$$ where $\overline{(-)}$ denotes the closure.
\end{prop}

\begin{proof}
We will show $\overline{\bigcup_{w\in W}wC_+}$ gives a half-space of $V^*$. 
Then, because $C_-=-C_+$, $\overline{\bigcup_{w\in W}wC_-}$ gives the rest half-space and we get the conclusion.

Let $M=(M(i,j))_{i,j\in\Delta_0}$ be the matrix defined by $M(i,j) :=-\frac{1}{2}(m_{ij}-\delta_{ij})$ and 
let $V^{\perp} :=\{v\in V\ |\ Mv=0\}$. 
Then $V^{\perp}$ is a 1-dimensional vector space given by $\sum_{i=1}^nc_i\alpha_i$, where $c_i>0$ for any $i$ \cite[Proposition 2.6]{Hu}.

Let $E := \{f\in V^*\ |\ \langle f, V^{\perp}\rangle = 1\}$, 
$Z_i := \{f\in V^*\ |\ \langle f, \alpha_i \rangle = 0\}$ and 
$E_i := E\cap Z_i$. 
For $v\in V^{\perp}$, 
we have $\sigma_w(v)=v$ for $w\in W$ \cite[Proposition 6.3]{Hu}, and hence $\sigma_w^*$ stabilizes $E$.
Then $\sigma_{s_i}^*$ acts as an orthogonal reflection relative to 
$E_i$ and 
 $E_i$ gives the geometric description of the affine Weyl group as the hyperplanes bounding the alcove $F:=C_+\cap E$ \cite[section 6.5]{Hu}. 
Therefore the action of $W$ to $F$  permutes the all alcoves in $E$ transitively \cite[Proposition 4.3]{Hu} and hence 
$\bigcup_{w\in W}w(F)=E$. 
Thus, $wC_+\cap E = w(F)$ and hence 
$\bigcup_{w\in W}wC_+\cap E=E$.
Consequently, 
$\overline{\bigcup_{w\in W}wC_+} = \{f\in V^* | \langle f, V^{\perp}\rangle \geq 0\}$ and we get the conclusion.
\end{proof}

\begin{exam}
\begin{itemize}
\item[(a)]
Let $\Delta$ be the type $\widetilde{\mathbb{A}}_1$. 
The situation can be described as follows.

$$
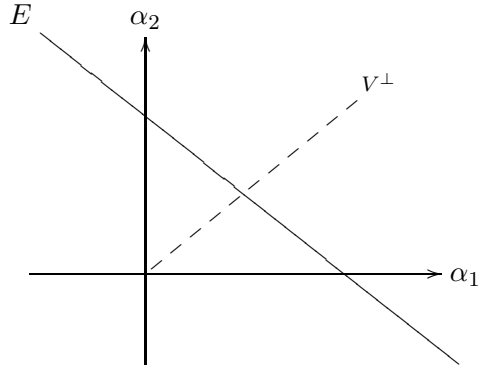
\xymatrix@C12pt@R10pt{  
E\ar@{-}[rrrrrrrddddddd]& &\alpha_2  & &&&&\\
 & & & &&&{\begin{smallmatrix}V^\perp\end{smallmatrix}}&\\
 & &&  &&&&\\
& && &&&&\\
& & & & &&&\\
\ar[rrrrrrr]^{} && 
\ar@{--}[uuuurrrr]
&  & &   &&\alpha_1 \\
 &&& &&&&\\
&&\ar[uuuuuuu]&&&&&}
$$

In this case, the intersection of $E$ and $\alpha_1$ (respectively,  $\alpha_2$) gives $E_2$ (respectively, $E_1$) and their  interval is $F$.

\item[(b)]
Let $\Delta$ be the type $\widetilde{\mathbb{A}}_2$.
The situation can be described as follows.
$$
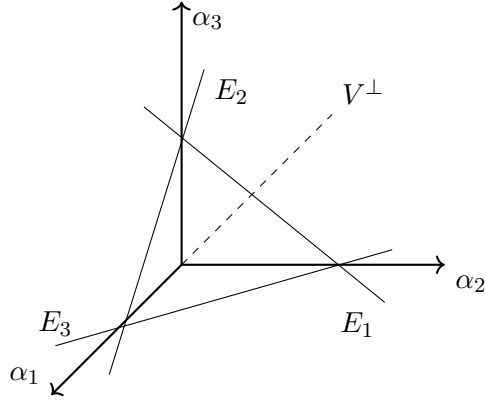
\begin{tikzpicture}
\draw[thick,->] (0,0,0) -- (3.5,0,0) node[anchor=north west]{$\alpha_2$};
\draw[thick,->] (0,0,0) -- (0,3.5,0)node[anchor=north west]{$\alpha_3$};
\draw[thick,->] (0,0,0) -- (0,0,4.5) node[anchor=south east]{$\alpha_1$};
\coordinate (O) at (0,0,0);
\coordinate (P) at (2,2,0);
\draw[dashed] (O) -- (P) node[anchor=south west]{$V^\perp$};
\draw  (-0.5,2.1,0) -- (2.7,-0.5,0) node[anchor=north east]{$E_1$};
\draw (0,-0.5,2.5)-- (0.3,2.6,0)    node[anchor=north west]{$E_2$};
\draw  (2.8,0.2,0) -- (-1.1,-0.5,1.5)  node[anchor=south ]{$E_3$};
\end{tikzpicture}
$$
\end{itemize}
In this case, the inner region surrounded by $E_1,E_2$ and $E_3$ gives $F$.
\end{exam}

Next we use the following result.

\begin{thm}\cite{DIJ,P,Hi}\label{injective}
\begin{itemize}
\item[(a)] The map $T \mapsto [T]$ induces an injection $\twosilt\La\to K_0(\KKb(\proj\La))$.
\item[(b)] Let $T$ and $U$ be two-term silting complexes in $\KKb(\proj\La)$. 
If $T\ncong U$, then $C(T)$ and $C(U)$ intersect only at their boundaries.
\end{itemize}
\end{thm}

\begin{proof}
(a) follows from \cite{DIJ}, together with \cite{P} in which infinite dimensional cases are discussed. 
For the convenience of the reader, we give a sketch of a proof.
Let 
\begin{align*}
T :=(\xymatrix@C10pt@R5pt{\cdots\ar[r]&0\ar[r]& T^1\ar[r]^{f^T}& T^0 \ar[r]^{}&0\ar[r]&\cdots}) \ \textnormal{and}\
U :=(\xymatrix@C10pt@R5pt{\cdots\ar[r]&0\ar[r]& U^1\ar[r]^{f^U}& U^0 \ar[r]^{}&0\ar[r]&\cdots})
\end{align*} 
 be two-term silting complexes.
We have $\add (T^{0}) \cap \add (T^{1})=\{0\}$ and $\add (U^{0}) \cap \add (U^{1})=\{0\}$ by \cite[Lemma 2.25]{AI}.
Assume that $[T]=[U]$.
This implies that $T^{0}\simeq U^{0}$ and $T^{1}\simeq U^{1}$.
Consider an action of the group $\Aut_\La(T^{0})\times\Aut_\La(T^{1})$ on $\Hom_\La(T^{1},T^{0})$ given by $(g^{0}, g^{1})(f):=g^{0}f(g^{1})^{-1}$.
Then by the same argument as \cite[subsection 3.1]{P}, the orbits of $f^{T}$ and $f^{U}$ in $\Hom_\La(T^{1},T^{0})$ intersect, which implies $T\cong U$.

(b) follows from (a) and \cite{Hi,DIJ}.
\end{proof}

Using the above results, we finally  obtain the following proof. 

\begin{proof}[Proof of Theorem \ref{classify}]
Let $S$ be a two-term silting complex $\KKb(\proj\La)$. Then $[S]$ gives a basis of $K_0(\KKb(\proj\La))$ \cite[Theorem 2.27]{AI}. 
On the other hand, by
Proposition \ref{cover}, we have an equality
$\overline{\bigcup_{w\in W}C(P_w)\cup\bigcup_{w\in W} C(R_w)} = V^*$. 
Thus Theorem \ref{injective} (b) shows that $S\cong P_w$ or $S\cong R_w$ for some $w\in W$. 
Because $\{P_w\}_{w\in W}$ and $\{R_w\}_{w\in W}$ are tilting complexes by Propositions \ref{P_w} and \ref{R_w}, we get the assertion.
\end{proof}


\section{The Krull-Schmidt property of homotopy categories}\label{Krull-Schmidt property}

Let $\La$ be a ring and $\fm$ be a two-sided ideal of $\La$.
Then $\La$ is said to be \emph{complete with respect to the $\fm$-adic topology} if the natural morphism $\La \to \varprojlim (\La/\fm^{i})$ is an isomorphism, where $\varprojlim (\La/\fm^{i})$ is the inverse limit of an inverse system $(\cdots \to \La/\fm^{i+1} \to \La/\fm^{i} \to \cdots)$ of rings.
We consider the following condition (F) on $\La$:
\begin{center} (F) $\La/\fm^{i}$ has finite length for all $i\geq 1$.
\end{center}

\begin{exam}\label{exam-pseudo-compact-ring}
	The following rings are examples of complete rings satisfying the condition (F).
	\begin{itemize}
		\item[(i)]
		A left artinian ring $\La$ with $\fm=\rad\La$.
		\item[(ii)]
		Let $R$ be a commutative local complete noetherian ring with the maximal ideal $\fm$.
		Then an $R$-algebra $\La$, which is finitely generated as an $R$-module, is a complete ring by a two-sided ideal $\fm\La$ satisfying the condition (F) by \cite[(6.5) Proposition]{CR}.
		\item[(iii)]
		For a finite quiver $Q$, let $KQ$ be the path algebra of $Q$ over a field $K$ and $I$ be a two-sided ideal of $KQ$.
		We denote by $\widehat{KQ}$ the complete path algebra of $Q$ and denote by $\mathfrak{n}$ the two-sided ideal of $\widehat{KQ}$ generated by all arrows.
		The closure $\overline{I}=\bigcap_{i=0}^{\infty}(I + \mathfrak{n}^{i})$ of $I$ with respect to the $\mathfrak{n}$-adic topology on $\widehat{KQ}$ is a two-sided ideal of $\widehat{KQ}$.
		Then a $K$-algebra $\La:=\widehat{KQ}/\overline{I}$ is a  complete ring by the ideal $\fm$ generated by all arrows and 
			$\fm$ satisfies the condition (F).
	\end{itemize}
\end{exam}

	Assume that $\La$ is complete with respect to the $\fm$-adic topology satisfying the condition (F).
	Then $\fm$ and the radical $\rad \La$ of $\La$ define the same adic topology on $\La$.
	Therefore, in this case, we say that $\La$ is complete instead of saying that $\La$ is complete with respect to the $\fm$-adic topology. A complete ring which satisfies the condition (F) is also called a pseudo-compact ring (we refer \cite{Ga, KeY, VdB} for details).

Recall that an additive category is called a \emph{Krull-Schmidt} category if every object decomposes into a finite direct sum of objects having local endomorphism rings.
Any object of a Krull-Schmidt category decomposes into a finite direct sum of indecomposable objects, and such a decomposition is unique up to isomorphisms and a permutation (see \cite[Corollary 4.3]{Krause} for details).
For a ring $R$, we denote by $\proj R$ the category of finitely generated projective left $R$-modules. 

In this section, we show that if $\La$ is a complete ring which  satisfies the condition (F), then the homotopy category $\KKb(\proj\La)$ is a Krull-Schmidt category.
The theorem was originally shown in the case where $\La$ is a complete path algebra in \cite{KeY}.
The authors thank Bernhard Keller and Dong Yang for their pointing out the validity of result for a complete ring which satisfies the condition (F).

It is easy to show that $R$ is complete by a two-sided ideal $\fm$ satisfying the condition (F)
if and only if $R$ satisfies the following three conditions: (i) $R/\rad R$ is a semi-simple ring (ii) $R$ is complete with respect to the $(\rad R)$-adic topology, and (iii) $\rad R$ is a finitely generated left $R$-module.

We first observe that if $\La$ is a complete ring which satisfies the condition (F), then $\proj\La$ is Krull-Schmidt. 
A ring $R$ is called \emph{semi-perfect} \cite[\S27]{AF} if it satisfies (i) $R/\rad R$ is a semi-simple ring and (ii) every idempotent in $R/\rad R$ is the image of an idempotent in $R$. 
Then the following lemma about semi-perfect rings is well-known.

\begin{lemm}\label{lem-semi-perfect-KS}
	Let $R$ be a ring.
	Then the following statements are equivalent.
	\begin{itemize}
		\item[(i)]
		$R$ is a semi-perfect ring.
		\item[(ii)]
		Every finitely generated $R$-module has a projective cover.
		\item[(iii)]
		The category $\proj R$ is a Krull-Schmidt category.
	\end{itemize}
\end{lemm}
\begin{proof}
	See \cite[27.6. Theorem]{AF} and \cite[Proposition 4.1]{Krause}.
\end{proof}

Then we observe the following lemma.

\begin{lemm}\label{lem-pseudo-compact-semi-perfect}
	Let $R$ be a ring and $\fm$ a two-sided ideal of $R$.
	Assume that $R$ is complete and $R/\fm$ is a left artinian ring.
	Then $R$ is semi-perfect and therefore $\proj R$ is Krull-Schmidt.
\end{lemm}
\begin{proof}
	Since $R$ is complete, for any $b\in\fm$, $1-b$ is an invertible element of $R$.
	Namely, $\fm$ is contained in $\rad R$.
	Since $R/\fm$ is a left artinian ring, $R/\rad R$ is a semi-simple ring and every idempotent in $R/\rad R$ is the image of an idempotent in $R/\fm$.
	By \cite[(6.7) Theorem (i)]{CR},	every idempotent in $R/\fm$ is the image of an idempotent in $R$.
	Thus $R$ is a semi-perfect ring.
	By Lemma \ref{lem-semi-perfect-KS}, $\proj R$ is Krull-Schmidt.
\end{proof}

Next, to state the main theorem, we consider the following setting.
For a ring $\Ga$, we denote by $\Mod\Ga$ the category of left $\Ga$-modules.
Let $\La$ be a subring of $\Ga$.
For a $\Ga$-module $M$, when we regard $M$ as a $\La$-module by an inclusion $\La\subset\Ga$, we write ${}_{\La}M$.
We denote by $\cC_{\La}^{\Ga}$ the full subcategory of $\Mod\Ga$ consisting of $\Ga$-modules $M$ which is finitely generated projective as a $\La$-module, that is,
$$
\cC_{\La}^{\Ga}:=\{ M \in \Mod\Ga \mid {}_{\La}M\in\proj\La \}.
$$

From now on to the end of this section,	assume that $\La$ is complete by a two-sided ideal $\fm$ satisfying the condition (F).
We show the following theorem in this section.

\begin{thm}\label{thm-cC-KS}
	Let $\La$ be a subring of a ring $\Ga$. 
	If $\Ga\fm \subset \fm\Ga$ holds, then $\cC_{\La}^{\Ga}$ is a Krull-Schmidt category.
\end{thm}


Before proving Theorem \ref{thm-cC-KS}, we apply this theorem to show that, for a complete ring $\La$ which satisfies the condition (F), the homotopy category $\KKb(\proj\La)$ of finitely generated projective $\La$-modules  is Krull-Schmidt.

The following lemma directly follows from the definition of Krull-Schmidt categories.

\begin{lemm}\label{lem-KS-quotient}
	Let $F : \mathcal{B} \to \mathcal{C}$ be a full dense additive functor between additive categories $\mathcal{B}, \mathcal{C}$.
	If $\mathcal{B}$ is Krull-Schmidt, then so is $\mathcal{C}$.
\end{lemm}

Then we have the following corollary of Theorem \ref{thm-cC-KS}.
We denote by $\CCb(\proj\La)$ the category of bounded complexes of $\proj\La$.

\begin{cor}\label{cor-CCb-KKb-KS}
	We have the following statements.
	\begin{itemize}
		\item[(a)]
		$\CCb(\proj\La)$ is a Krull-Schmidt category.
		\item[(b)]
		$\KKb(\proj\La)$ is a Krull-Schmidt category.
	\end{itemize}
\end{cor}

\begin{proof}
	(a)
	For an integer $n\geq 1$, we denote by $\mathsf{T}_{n}(\La)$ the $n\times n$ lower triangular matrix ring over $\La$ and denote by $E_{i,j}$ the $(i,j)$-matrix unit of $\mathsf{T}_{n}(\La)$.
	If $n=1,2$, then let $\Ga_{n}:=\mathsf{T}_{n}(\La)$.
	If $n\geq 3$, then let $\Ga_{n}$ be the factor ring of $\mathsf{T}_{n}(\La)$ modulo the ideal generated by $E_{i+2,i}$ for $i=1,\ldots,n-2$, that is, 
	
	$$\Ga_{n}=\left(
\begin{array}{cccc}
\La&&&O\\
\La&\La&&\\
&\ddots&\ddots&\\
O&&\La&\La
\end{array}
\right).$$
	
	Then $\La$ is a subring of $\Ga_{n}$ and $\fm\Ga_{n}=\Ga_{n}\fm$ holds.
	By Theorem \ref{thm-cC-KS}, $\cC_{\La}^{\Ga_{n}}$ is a Krull-Schmidt category.
	It is easy to show that there exists a fully faithful functor
	\begin{align*}
	\rho_{n} : \cC_{\La}^{\Ga_{n}} \longrightarrow \CCb(\proj\La),
	\end{align*}
	given by $M \mapsto (e_{1}M \xto{f_{1}} e_{2}M \xto{f_{2}} \cdots \xto{f_{n-1}} e_{n}M)$, where $e_{i}$ is an idempotent of $\Ga_{n}$ associated to $E_{i,i}$ and $f_{i}$ is the multiplication map of $E_{i+1,i}$ from the left.
	Then it is also easy to show that any object $X\in\CCb(\proj\La)$ is isomorphic to $\rho_{n}(M)$ for some integer $n\geq 1$ and some $M\in\cC_{\La}^{\Ga_{n}}$ up to shift.
	Since $\cC_{\La}^{\Ga_{n}}$ is Krull-Schmidt, $X$ decomposes into a finite direct sum of objects having local endomorphism rings.
	Therefore, $\CCb(\proj\La)$ is a Krull-Schmidt category.
	
	(b)
	By the definition of homotopy categories, there exists a full dense functor from $\CCb(\proj\La)$ to $\KKb(\proj\La)$.
	Thus $\KKb(\proj\La)$ is a Krull-Schmidt category by (a) and Lemma \ref{lem-KS-quotient}.
\end{proof}

From now on, we show Theorem \ref{thm-cC-KS}.
We begin with the following lemma.

\begin{lemm}\label{lem-cC-decompose}
	Let $\La$ be a subring of a ring $\Ga$.
	Assume that $\proj\La$ is a Krull-Schmidt category.
	Then each object of $\cC_{\La}^{\Ga}$ decomposes into a finite direct sum of indecomposable objects in $\cC_{\La}^{\Ga}$.
\end{lemm}

\begin{proof}
	Let $M\in\cC_{\La}^{\Ga}$.
	Since $\proj\La$ is Krull-Schmidt, ${}_{\La}M$ uniquely decomposes into a finite direct sum of indecomposable projective $\La$-modules.
	We show the assertion by an induction of the number of indecomposable direct summands of ${}_{\La}M$.

	It is clear that, for $M, N \in \cC_{\La}^{\Ga}$, if $M\simeq N$, then ${}_{\La}M \simeq {}_{\La}N$ holds.
	If $M$ is indecomposable, there is nothing to show. 
	Assume that $M$ and hence ${}_{\La}M$ are decomposable, and the number of indecomposable direct summands of ${}_{\La}M$ is $n$.
	Then we have $M\simeq M^{\prime}\oplus M^{\prime\prime}$ for some $M^{\prime}, M^{\prime\prime}\in\cC_{\La}^{\Ga}$ and the numbers of indecomposable direct summands of ${}_{\La}M^{\prime}$ and ${}_{\La}M^{\prime\prime}$ are smaller than $n$.
	By the induction hypothesis, $M^{\prime}$ and $M^{\prime\prime}$ decompose into finite direct sums of indecomposable objects in $\cC_{\La}^{\Ga}$.
\end{proof}

The next lemma is well-known as Fitting-Lemma.

\begin{lemm}\label{lem-Fitting}
	Let $R$ be a ring and $M$ be a finite length $R$-module.
	Then for each $R$-morphism $f : M \to M$, there exists an integer $n>0$ such that $M\simeq\Im(f^{n})\oplus \Ker(f^{n})$.
	In this case, we have $\Im(f^{n})=\Im(f^{n+i})$ and $\Ker(f^{n})=\Ker(f^{n+i})$ for any $i>0$.
\end{lemm}

Then we give the following key lemma, which is a generalization of Lemma \ref{lem-Fitting}.

\begin{lemm}\label{lem-end-local-CCb}
	Let $\La$ be a subring of a ring $\Ga$.
	If $\Ga\fm \subset \fm\Ga$ holds,
	then for each $M\in\cC_{\La}^{\Ga}$ and each $f\in\Hom_{\Ga}(M,M)$, there exist $I, K\in\cC_{\La}^{\Ga}$ such that $M\simeq I \oplus K$, $f(I)\subset I$, $f(K)\subset K$, $f|_{I}$ is an isomorphism on $I$ and $({\rm id}_{K}-f|_{K})$ is an isomorphism on $K$.
	In particular, if $M$ is indecomposable, then its endomorphism algebra $\Endm_{\Ga}(M)$ is local.
\end{lemm}
\begin{proof}
	Let $M\in\cC_{\La}^{\Ga}$ and $f\in\Hom_{\Ga}(M,M)$ and $i > 0$ be an integer.
	Since $\Ga\fm \subset \fm\Ga$ holds, $M/\fm^{i}M$ is a left $\Ga$-module, and $f$ induces a morphism $f_{i} : M/\fm^{i}M \to M/\fm^{i}M$ of $\Ga$-modules.
	Thus we have the following commutative diagram
	\begin{equation*}
	\begin{tikzcd}
	\cdots \ar[r] & M/\fm^{i+1}M \ar[r, "p_{i+1}"] \ar[d, "f_{i+1}"] & M/\fm^{i}M \ar[r, "p_{i}"] \ar[d, "f_{i}"] & M/\fm^{i-1}M \ar[r] \ar[d, "f_{i-1}"] & \cdots \\
	\cdots \ar[r] & M/\fm^{i+1}M \ar[r, "p_{i+1}"] & M/\fm^{i}M \ar[r, "p_{i}"] &M/\fm^{i-1}M \ar[r] & \cdots,
	\end{tikzcd}
	\end{equation*}
	where $p_{i}$ is a canonical morphism.
	Since $\La$ satisfies the condition (F) and $M$ is in $\cC_{\La}^{\Ga}$, $M/\fm^{i}M$ has finite length as a $\La$-module.
	By Lemma \ref{lem-Fitting}, there exists an integer $n_{i}>0$ satisfying the following equalities and an isomorphism
	\begin{align}
	\Im((f_{i})^{n_{i}})=\Im((f_{i})^{n_{i}+j}) \label{equ-Im} \\
	\Ker((f_{i})^{n_{i}})=\Ker((f_{i})^{n_{i}+j}) \label{equ-Ker} \\
	M/\fm^{i}M \simeq \Im( (f_{i})^{n_{i}})\oplus \Ker( (f_{i})^{n_{i}}) \notag
	\end{align}
	for any integer $j>0$.
	Define $\Ga$-modules $I_{i}:=\Im((f_{i})^{n_{i}})$ and $K_{i}:=\Ker((f_{i})^{n_{i}})$.
	By equalities (\ref{equ-Im}) and (\ref{equ-Ker}), we have a commutative diagram
	\begin{equation}\label{CommDiagramf}
	\begin{tikzcd}
	M/\fm^{i}M \ar[r, "\simeq"] \ar[d, "f_{i}"] & I_{i} \oplus K_{i} \ar[d, "f_{i}\mid_{I_i}\oplus f_{i}\mid_{K_i}"] \\
	M/\fm^{i}M \ar[r, "\simeq"] & I_{i} \oplus K_{i}.
	\end{tikzcd}
	\end{equation}
	Note that $f_{i}\mid_{I_i}$ is an isomorphism and $f_{i}\mid_{K_i}$ is nilpotent.
	
	We can assume that $n_{i+1}>n_{i}$ for any $i>0$.
	Thus there exist natural morphisms $\phi_{i+1} : I_{i+1} \to I_{i}$ and $\psi_{i+1} : K_{i+1} \to K_{i}$, and we have inverse systems $(I_{i}, \phi_{i})$ and $(K_{i}, \psi_{i})$.
	Now we have the following commutative diagram
	\begin{equation}\label{CummDiagramp}
	\begin{tikzcd}
	M/\fm^{i}M \ar[r, "\simeq"] \ar[d, "p_{i}"] & I_{i} \oplus K_{i} \ar[d, "\phi_{i}\oplus \psi_{i}"] \\
	M/\fm^{i-1}M \ar[r, "\simeq"] & I_{i-1} \oplus K_{i-1}.
	\end{tikzcd}
	\end{equation}
	Since ${}_{\La}M$ is a finitely generated projective $\La$-module, an inverse system $(M/\fm^{i}M, p_{i})$ satisfies $M \simeq \varprojlim(M/\fm^{i}M, p_{i})$.
	Let $I:=\varprojlim(I_{i}, \phi_{i})$ and $K:=\varprojlim(K_{i}, \psi_{i})$.
	By the commutative diagram (\ref{CummDiagramp}), we have $M\simeq I \oplus K$.
	It is easy to see that $\cC_{\La}^{\Ga}$ is closed under direct summands  in $\Mod\Ga$.
	Thus $I$ and $K$ belong to $\cC_{\La}^{\Ga}$.
	By the commutative diagram (\ref{CommDiagramf}), $f|_{I}=\varprojlim(f_{i}\mid_{I_i})$ and  $f|_{K}=\varprojlim(f_{i}\mid_{K_i})$ hold.
	Thus we have $f(I)\subset I$ and $f(K)\subset K$.
	Since $f_{i}\mid_{I_i}$ is an isomorphism for any $i>0$, $f|_{I}$ is also an isomorphism.
	Similarly since $(\id\mid_{K}-f_{i}\mid_{K_i})$ is an isomorphism for any $i>0$, $(\id_{K}-f|_{K})$ is also an isomorphism.
\end{proof}

Then Theorem \ref{thm-cC-KS} directly follows from Lemmas \ref{lem-cC-decompose} and \ref{lem-end-local-CCb}.

We end this section by giving one observation.
Let $R$ be a ring.
We denote by $\fp R$ the full subcategory of $\Mod R$ consisting of finitely presented $R$-modules.
Then we have the following result. 
\begin{lemm}\label{lem-fp-KS}
	Let $R$ be a ring.
	If $\KKb(\proj R)$ is Krull-Schmidt, then so is $\fp R$.
\end{lemm}

\begin{proof}
	Let $\mathcal{B}$ be a full subcategory of a Krull-Schmidt category $\mathcal{C}$.
	It is easy to see that if $\mathcal{B}$ is closed under direct sums and direct summands in $\mathcal{C}$, then $\mathcal{B}$ is a Krull-Schmidt category.
	
	We denote by $\mathcal{K}$ the subcategory of $\KKb(\proj R)$ consisting of two-term complexes.
	For any $X\in\KKb(\proj R)$, $X$ belongs to $\mathcal{K}$ if and only if $\Hom(R,X[i])=0$ and $\Hom(X,R[i-1])=0$ for any $i>0$.
	Because of this characterization, $\mathcal{K}$ is closed under direct sums and direct summands in $\KKb(\proj R)$.
	Therefore by the above argument, $\mathcal{K}$ is Krull-Schmidt.
	We denote by $H^{0} : \mathcal{K} \to \fp R$ the functor which takes the degree zero homology.
	Then it is easy to see that $H^{0}$ is full and dense.
	Thus by Lemma \ref{lem-KS-quotient}, $\fp R$ is Krull-Schmidt.
\end{proof}

\appendix
\section{Rings whose silting complexes are tilting \\ by Osamu Iyama}

The aim of this section is to prove the following observation.

\begin{prop}\label{silt=tilt2}
	Let $\Lambda$ be a preprojective algebra (or, complete preprojective algebra) of affine type over an arbitrary field. Then all silting complexes of $\Lambda$ are tilting.
\end{prop}

To prove this, we need the following generalization of a well-known fact that all silting complexes of a finite dimensional symmetric algebras over a field are tilting.

\begin{prop}\label{silt=tilt}
	Let $R$ be a commutative Cohen-Macaulay ring with canonical module $\omega$, and $\Lambda$ a module-finite $R$-algebra such that
	\begin{equation}\label{sCY}
	\RHom_R(\Lambda,\omega)\simeq\Lambda
	\end{equation}
	in the derived category of $\Lambda$-bimodules. Then all silting complexes of $\Lambda$ are tilting.
\end{prop}

Notice that the condition \eqref{sCY} is equivalent to the following conditions.
\begin{itemize}
	\item $\Hom_R(\Lambda,\omega)\simeq\Lambda$ as $\Lambda$-bimodules.
	\item $\Ext^i_R(\Lambda,\omega)=0$ for all $i>0$, that is, $\Lambda$ is a maximal Cohen-Macaulay $R$-module.
\end{itemize}

\begin{remk}
	When $R$ is a Gorenstein ring and $\omega=R$, an $R$-algebra $\Lambda$ satisfying \eqref{sCY} is called a \emph{symmetric $R$-order}. This is a special case of a \emph{singular Calabi-Yau algebra} \cite{IW1} (=Calabi-Yau$^-$ algebra \cite{IR1}).
	
	Under certain technical conditions on $R$, the condition \eqref{sCY} is equivalent to being a \emph{bimodule} \emph{Calabi-Yau algebra}  \cite[Theorem 7.2.14]{Gi}.
\end{remk}

We start with proving Proposition \ref{silt=tilt}.

\begin{proof}[Proof of Proposition \ref{silt=tilt}]
	Let $T$ be a silting complex of $\La$, and $\EE=\RHom_{\La}(T,T)$. Then $\EE\in\DD^{\le0}(R)$ holds, where $(\DD^{\le0}(R),\DD^{\ge0}(R))$ is the standard t-structure of the derived category $\DD(R)$ of the category of $R$-modules.
	On the other hand, by our assumption \eqref{sCY}, there is a functorial isomorphism
	\[\RHom_\Lambda(X,Y)\simeq\RHom_R(\RHom_\Lambda(Y,X),\omega)\]
	in $\DD(R)$ for all $X\in\DD(\Lambda)$ and $Y\in\KKb(\proj\Lambda)$, see \cite[Proposition 3.5(2)(3)]{IR1}.
	Evaluating $X=Y=T$, we obtain an isomorphism
	\[\EE\simeq\RHom_R(\EE,\omega)\]
	in $\DD(R)$.
	Since $\EE\in\DD^{\le0}(R)$, we have $\RHom_R(\EE,\omega)\in\DD^{\ge0}(R)$. Thus $\EE\in\mod R$ holds, that is, $T$ is a tilting complex.
\end{proof}

Now we are able to prove Proposition \ref{silt=tilt2}.

\begin{proof}
	It is well-known that the center $R$ of $\Lambda$ is a simple singularity in dimension 2, and $\Lambda$ is a module-finite $R$-algebra satisfying \eqref{sCY} \cite{GL1,GL2}. Thus the assertion follows from Proposition \ref{silt=tilt}.
\end{proof}


\end{document}